\documentclass[12pt]{amsart} 

\usepackage{amssymb, latexsym}
\usepackage{xcolor}

\newtheorem{theorem}{Theorem}[section]

\newtheorem{lemma}[theorem]{Lemma}

\usepackage{enumitem}
\usepackage{graphicx}

\numberwithin{equation}{section}

\begin{document}

\title[Zeros of the extended Selberg class zeta-functions] {Zeros of the extended Selberg class zeta-functions and of their derivatives  
}

\author{ Ram\={u}nas Garunk\v{s}tis}
\address{Ram\={u}nas Garunk\v{s}tis \\
Faculty of Mathematics and Informatics, Institute of Mathematics, Vilnius University \\
Naugarduko 24, 03225 Vilnius, Lithuania}
\email{ramunas.garunkstis@mif.vu.lt}
\urladdr{www.mif.vu.lt/~garunkstis}

\subjclass[2010]{Primary: 11M26; Secondary: 11M41}

\keywords{Riemann zeta-function; extended Selberg class; nontrivial zeros; Speiser's equivalent for the Riemann hypothesis}

\begin{abstract} 
Levinson and Montgomery proved that the
Riemann zeta-function $\zeta(s)$  and its derivative  have approximately the same number of non-real zeros  left of the critical line. 
R. Spira showed that $\zeta'(1/2+it)=0$ implies $\zeta(1/2+it)=0$. Here we obtain that in  small areas located to the left of the critical line and near it the  functions $\zeta(s)$ and $\zeta'(s)$ have the same number of zeros.  We prove our result for more general zeta-functions from the extended Selberg class $S$. We also consider zero trajectories of a certain family of  zeta-functions from $S$.
\end{abstract}

\maketitle

\section{Introduction}

Let $s=\sigma+it$.  In this paper, $T$ always tends to plus infinity.

Speiser \cite{Speiser1934} showed that the Riemann hypothesis (RH) is
equivalent to the absence of non-real zeros of the derivative of the
Riemann zeta-function $\zeta(s)$ left of the critical line $\sigma=1/2$. Later on, Levinson and Montgomery \cite{Levinson1974} obtained the quantitative version of the
Speiser's result:

\bigskip

\noindent {\bf Theorem (Levinson-Montgomery)}\, {\it Let $N^-(T)$ be the number of zeros of $\zeta(s)$ in $R: 0<t<T, 0<\sigma<1/2$. Let $N^-_1(T)$ be the number of zeros of $\zeta'(s)$ in $R$. Then
$N^-_1(T)=N(T)+O(\log T).$ 

Unless $N^-(T)>T/2$ for all large $T$ there exists a sequence $\{T_j\}$, $T_j\to\infty$ as $j\to\infty$ such that
$N^-_1(T_j)=N^-(T_j).$
}
\bigskip

Here we prove the following theorem.
\begin{theorem}\label{dertest}
 There is an absolute constant
\(T_0>0\)  such that, for any 
\(T>T_0\) and any \(A>0.17\), there is a radius \(r\),

\begin{equation*}
\exp(-T^A)\le r\le\exp(-T^{A-0.17}),
\end{equation*}
 such that in the region

\begin{equation*}
\{s : |s-(1/2+iT)|\le r\ {\text{and}}\ \sigma<1/2\}
\end{equation*}
the functions \(\zeta(s)\) and \(\zeta'(s)\) have the same number of zeros.
\end{theorem}
 
 Non-real zeros of $\zeta(s)$ lie symmetrically with respect to the critical line. In  this sense, the result of Spira \cite[Corollary 3]{Spira69} that $\zeta(1/2+it)=0$ if $\zeta'(1/2+it)=0$ can be regarded as a border case of the above theorem when 
$r=0$. 

Note that for both $\zeta(s)$ and \(\zeta'(s)\) the average gap between zeros is $2\pi/\log T$ around height $T$ (Titchmarsh \cite[Section 9.4]{Titchmarsh1986} and Berndt \cite{Berndt1970}). This is much larger than the radius $r$ in Theorem \ref{dertest}.

In Theorem \ref{dertest}, the constant $0.17$ is related to the  number of zeros of $\zeta(s)$ in the strip $|t-T|\le1/T$. For details see  Section \ref{proofs} which contains the proof  of Theorem~\ref{dertest}. Moreover, in Section \ref{proofTh4} we consider a more general  version of Theorem~\ref{dertest} devoted to the extended Selberg 
class $S$. The extended Selberg class  contains most of the classical $L$-functions (Kaczorowski \cite{Kaczorowski06}). This class also includes zeta-functions for which RH is not true, a well-known example  being the Davenport-Heilbronn
zeta-function, which is defined as a suitable linear combination of two Dirichlet $L$-functions (Titchmarsh \cite[Section
10.25]{Titchmarsh1986}, see also Kaczorowski and Kulas
\cite{Kaczorowski07}). In the next section we also investigate zero trajectories of the following family of zeta-functions from  $S$:

\begin{align}\label{fstau}
  f(s,\tau) := (1 - \tau)(1 + \sqrt{5}/5^s) \zeta (s)  + \tau L (s, \psi),
\end{align}
where $\tau \in [0, 1]$  and $L (s, \psi)$ is the Dirichlet $L$-function
with the Dirichlet character $\psi\bmod 5$, $\psi(2)=-1$. 

\section{Extended Selberg class}\label{proofTh4}

 We consider Theorem \ref{dertest} in the broader context of the extended Selberg class. Note that Levinson and Montgomery's \cite[Theorem 1]{Levinson1974}  approach, which is used here, usually works for zeta-functions having nontrivial zeros  distributed symmetrically with respect to the critical line. See Y{\i}ld{\i}r{\i}m \cite{yildirim96} for Dirichlet $L$-functions; \v Sle\v zevi\v cien\. e \cite{rasa} for the
Selberg class; Luo \cite{Luo2005},  Garunk\v stis \cite{Garunkstis2008}, Minamide \cite{Minamide2009}, \cite{Minamide2010}, \cite{Minamide2013}, Jorgenson and Smailovi\' c \cite{js} for  Selberg
zeta-functions and related functions; Garunk\v stis and \v Sim\.enas \cite{gs} for the extended Selberg class.
 In Garunk\v stis and Tamo\v si\= unas \cite{garunkstistamosiunas} the Levinson and Montgomery result was generalized to the Lerch zeta-function with equal parameters. Such function has an almost symmetrical distribution of non-trivial zeros with respect to the line $\sigma=1/2$. Insights which helped to overcome difficulties raised by ``almost symmetricity" in  \cite{garunkstistamosiunas}   led to Theorem \ref{dertest} of this paper, although $\zeta(s)$  has a strictly symmetrical zero-distribution.

We recall the definition of the extended Selberg class (see \cite{Kaczorowski06, Kaczorowski99, 
Steuding07}).  A not identically vanishing Dirichlet series

\begin{equation*}
 F(s)=\sum_{n=1}^{\infty} \frac{a_n}{n^s},
  \end{equation*}
  which converges absolutely for $\sigma>1$, belongs to the {\it
    extended Selberg class} $S$ if
  \renewcommand{\labelenumi}{(\roman{enumi}) }
  \begin{enumerate}
  \item (Meromorphic continuation) There exists $k\in\mathbb N$ such
    that $(s-1)^k F(s)$ is an entire function of finite order.
  \item (Functional equation) $ F(s)$ satisfies the functional
    equation:
    \begin{equation}\label{eq:selbergfunctional}
      \Phi(s) = \omega \overline{\Phi(1 - \overline{s})},
    \end{equation}
    where $\Phi(s) : = F(s) Q^s \prod_{j = 1}^r \Gamma(\lambda_j s +
    \mu_j)$, with $Q>0$, $\lambda_j > 0$, $\Re(\mu_j) \geq 0$ and
    $|\omega| = 1$.
  \end{enumerate}
  The data $Q$, $\lambda_j$, $\mu_j$ and $\omega$ of the
  functional equation are not uniquely determined by $F$, but the value
  $d_{ F} = 2 \sum_{j=1}^r \lambda_j$ is an invariant. It is called
  the \emph{degree} of $F$.

If the element of   $S$  also satisfies the Ramanujan hypothesis  ($a_n\ll_\varepsilon n^\varepsilon$ for any $\varepsilon>0$) and has a certain Euler product, then it belongs to the Selberg class introduced by Selberg \cite{Selberg92}.

We collect several properties of  $F(s)\in S$. The functional equation \eqref{eq:selbergfunctional} gives, for $F(1/2+it)\ne0$,

\begin{equation*}
\Re \frac{F'}{F}(1/2+it)=-\Re\sum_{j = 1}^{r} \lambda_j\frac{\Gamma'}{\Gamma}(\lambda_j (1/2+it) +
  \mu_j)-\log Q.
\end{equation*}
Then by the formula

\begin{equation*}
  \frac{\Gamma'}{\Gamma}(s) = \log s
  + O\left(|s|^{-1}\right)\quad( \Re(s) \geq 0, \ |s|\to\infty)
\end{equation*}
we get, for $F(1/2+it)\ne0$ and $d_F>0$, 
\begin{equation}\label{logr11/2}
\Re \frac{F'}{F}(1/2+it)=-\frac {d_F}2\log t -\log Q+O\left(\frac1t\right)\qquad(t\to\infty),
\end{equation}
where the implied constant may depend only on $\lambda_j$, $\mu_j$, $j=1,\dots, r$.

Every $F\in S$ has a zero-free half-plane, say $\sigma>\sigma_F$. By the functional equation, $F(s)$ has no zeros for $\sigma<-\sigma_F$, apart from possible trivial zeros coming from the poles of the $\Gamma$-factors. Let $\rho=\beta+i\gamma$ denote a generic zero of $F(s)$ and
 
\begin{equation*}
N_F(T)=\#\left\{\rho : F(\rho)=0, |\beta|\le\sigma_F, |\gamma|<T\right\}.
\end{equation*}
Then (Kaczorowski and Perelli \cite[Section 2]{Kaczorowski99})
\begin{equation}\label{VonMangoldt}
N_F(T)=\frac{d_F}{\pi}T\log T+c_FT+O(\log T)
\end{equation}
with a certain constant $c_F$, for any fixed $F\in S$ with $d_F>0$.

From the Dirichlet series expression for $F$ we see that there are constants $\sigma_1=\sigma_1(F)>1$ and $c=c(\sigma_1, F)>0$ such that 
\begin{equation}\label{gec}
|F(\sigma_1+it)|\ge c, \qquad t\in\mathbb R.
\end{equation}

 It is known (Garunk\v{s}tis and \v Sim\.enas \cite[formula (12)]{gs}) that  there is   $B=B(F,\sigma_1)>0$ such that
\begin{equation}\label{TB}
   |F(\sigma+iT)|<T^B, \quad (T>10),
\end{equation}
 for $\sigma\ge-4\sigma_1$. The specific constant $-4\sigma_1$ will be useful in the proof of Theorem \ref{prop} below.

In view of above, for given positive constants $\sigma_1$, $c$, $B$, $\varepsilon$, $\delta$,  $\bar T$, $\lambda_j$, and complex constants $\mu_j$ ($\Re(\mu_j)>0$, $j=1,\dots, r$),  we define a subclass $\bar{S}\subset S$ as the following: it  consists of functions satisfying \eqref{gec}, \eqref{TB}, \eqref{eq:selbergfunctional}, with any $|\omega|=1$; we require that any  function from $\bar{S}$ has no more than
\begin{equation}\label{1log2}
\frac{\varepsilon }{\log (2+\delta)} \log T-2
\end{equation}
zeros in the area $|t-T|\le1/T$, $T>\bar T$. For each function from $S$  the  Riemann-von Mangoldt type formula \eqref{VonMangoldt} yields the existence of  $\varepsilon$, $\delta$, and $\bar T$ such that the zero number bound \eqref{1log2} is fulfilled.

Theorem \ref{dertest} will be derived from the following more general statement. 

\begin{theorem}\label{prop}
Let $F(s)$ be an element of $\bar S$ with $d_F>0$.  Then there is a constant $T_0=T_0(\bar S)>0$ for which the following statement is true.

 If $A$ and 
  $s_0=\sigma_0+iT$ satisfy the inequalities 

\begin{equation}\label{ae}
A>\varepsilon, \quad T>T_0, \quad 1/2-\exp(-T^A)<\sigma_0\le1/2,
\end{equation}
 then there is a radius $r=r(F)$,

\begin{equation*}
\exp(-T^A)\le r\le\exp(-T^{A-\varepsilon}),
\end{equation*}
 such that in the area

\begin{equation}\label{area}
\{s : |s-s_0|\le r\ {\text{and}}\ \sigma<1/2\}
\end{equation}
functions $F(s)$ and $F'(s)$ have the same number of zeros.
\end{theorem}
Note that in Theorem \ref{prop} the constant $T_0$ is independent of $A$ and $\sigma_0$. This will be important in the proof of Theorem \ref{cor} below.

In  \cite{gs}   zeta-functions $f(s,\tau)$ defined by \eqref{fstau} were considered.
By Kaczorowski and Kulas
\cite[Theorem 2]{Kaczorowski07} we have that for any $\tau$ and any interval $(a,b)\subset (1/2,1)$ the function $f(s,\tau)$ has infinitely many zeros in the half-strip $a<\sigma<b$, $t>0$. 
Let $\theta>0$ and let

$$\rho : (\tau_0-\theta, \tau_0+\theta) \to\mathbb C$$
 be a continuous function such that $f(\rho(\tau), \tau)=0$ for $\tau\in(\tau_0-\theta, \tau_0+\theta)$. We say that $\rho(\tau)$ is a zero trajectory of the function $f(s, \tau)$. Analogously we define  a zero trajectory  $\tilde{\rho}(\tau)$ of the derivative $f'_s(s, \tau)$. See also the discussion below the formula (6) in \cite{gs}.
In \cite{gs}  several zero trajectories $\rho(\tau)$ of   $ f(s,\tau)$ and  $\tilde{\rho}(\tau)$ of  $ f'_s(s,\tau)$  were  computed. The behavior of these zero trajectories correspond well to Theorem \ref{prop}. 
Computations in \cite{gs} should be considered as heuristic because the accuracy was not controlled explicitly. Next we present a rigorous statement concerning  zero trajectories of $ f(s,\tau)$ and  $ f'_s(s,\tau)$.

\begin{theorem}\label{cor}
Let $\tau_0 \in [0, 1]$. Let $s=\rho_0$ be a second order zero of $f(s)=f(s,\tau_0)$ with $\Re(\rho_0)=1/2$ and sufficiently large $\Im(\rho_0)$. Then the following two statements are equivalent.
\renewcommand{\labelenumi}{\arabic{enumi}.}
\renewcommand{\labelenumii}{(\roman{enumii})}
\begin{enumerate}
\item[1)] There is   a zero trajectory $\rho(\tau)$, $\tau\in(\tau_0-\theta, \tau_0+\theta)$, $\theta>0$, of $f(s, \tau)$ such that
\begin{enumerate}
\item $\rho(\tau_0)=\rho_0$;
\item 
$\Re(\rho(\tau))=1/2$ if $\tau<\tau_0$; 
\item $\Re(\rho(\tau))<1/2$, if $\tau>\tau_0$.
\end{enumerate}
\item[2)] There is   a zero trajectory $\tilde{\rho}(\tau)$, $\tau\in(\tau_0-\eta, \tau_0+\eta)$, $\eta>0$, of $f'_s(s, \tau)$ such that
\begin{enumerate}
\item $\tilde{\rho}(\tau_0)=\rho_0$;
\item $\Re(\tilde{\rho}(\tau))>1/2$ if $\tau<\tau_0$; 
\item $\Re(\tilde{\rho}(\tau))<1/2$, if $\tau>\tau_0$.
\end{enumerate}
\end{enumerate}
\end{theorem}
From the proof  we see that Theorem \ref{cor} remains true if all inequalities $\tau<\tau_0$ and $\tau>\tau_0$ are simultaneously replaced by opposite inequalities. 

According to computations of \cite{gs} there are 1452 zero trajectories $\rho(\tau)$ of $f(s,\tau)$ with $0<\Im \rho(0)\le 1500$, 1166 of these trajectories stay on the critical line, while the remaining 286 leave it. The points at which mentioned trajectories leave the critical line are double zeros of $f(s)=f(s,\tau)$ (see also a discussion at the end of Section 3 in Balanzario and S\'{a}nchez-Ortiz \cite{bs2007}). In view of this we expect that the family $f(s,\tau)$, $\tau\in[0,1]$ has infinitely many double zeros lying on the line $\sigma=1/2$.
Moreover, we think that the similar statement to Theorem \ref{cor} can also be proved in the case where $s=\rho$ 
is a higher order zero of $f(s)=f(s,\tau)$ with $\Re \rho=1/2$, however there is no evidence such zeros exist.

The next section is devoted to the proofs of Theorems \ref{dertest}, \ref{prop}, and \ref{cor}.

\section{Proofs}\label{proofs}

Proof of Theorem \ref{prop} is based on the next lemma. Recall that the subclass $\bar S$ depends on constants $\sigma_1$, $c$, $B$, $\varepsilon$, $\delta$, $\bar T$, $\lambda_j$, $\mu_j$, ($j=1,\dots, r$).
		  		
\begin{lemma}\label{crho}
Let $F(s)$ be an element of $\bar S$ with $d_F>0$. Suppose that $s_0=\sigma_0+iT$ satisfies the inequality $\quad 1/2-\exp(-T^A)<\sigma_0\le1/2$, where $T>\bar T$ and  $A>\varepsilon$.
 Then there is a radius $r=r(F)$,
\begin{equation}\label{expta}
\exp(-T^A)\le r\le\exp(-T^{A-\varepsilon}),
\end{equation}
 such that,
for $|s-s_0|=r$, $\sigma\le1/2$,
 \begin{equation}\label{Nm}
\Re \frac{F'}{F}(s)\le -\frac{d_F}2\log T -\log Q+O\left(\frac1T\right),
\end{equation}
uniformly for $F(s)\in \bar S$.
\end{lemma}

\begin{proof}
We repeat the steps of the proof of Proposition 4 in \cite{garunkstistamosiunas}. Contrary to  Proposition 4, here   we do not need the upper bound for $\varepsilon$ (see \eqref{1log2}). This is because the ``symmetric" functional equation \eqref{eq:selbergfunctional} leads to the convenient formula \eqref{logr11/2}, while the ``almost symmetric" functional equation of the Lerch zeta-function with equal parameters in \cite{garunkstistamosiunas} leads to a more restricted version of \eqref{logr11/2} (see \cite[Lemma 3]{garunkstistamosiunas}).

Let $T>\bar T$ and $r_k=\exp\left(-(2+\delta)^{-k}T^A\right)$, $k=1,\dots,[\frac{\varepsilon }{\log (2+\delta)} \log T]$.  
By \eqref{1log2} and Dirichlet's box principle there is $j=j(F)\in\{2,\dots,[\frac{\varepsilon }{\log (2+\delta)} \log T]\}$ such that the region
\begin{equation}\label{ring}
r_{j-1}<|s-s_0|\le r_j
\end{equation}
has no zeros of $F(s)$. Then the auxiliary function
\begin{equation}\label{FF}
g(s):=\frac{F'}{F}(s)-\sum_{\rho\, :\, |\rho-s_0|\le r_{j-1}}\frac{1}{s-\rho}
\end{equation}
is analytic in the disc $|s-s_0|\le r_j$ and in this disc we have
\begin{equation}\label{Cauchy}
g(s)=\sum_{n=0}^\infty a_n(s-s_0)^n\quad\text{and}\quad a_n=\frac1{2\pi i}\int\limits_{|s-s_0|= r_j}\frac{g(s)ds}{(s-s_0)^{n+1}}.
\end{equation}

In view of bounds \eqref{gec} and \eqref{TB},  Lemma $\alpha$ from Titchmarsh \cite[Section 3.9]{Titchmarsh1986} gives that, for $|s-s_0|\le r_j$,

\begin{equation*}
\frac{F'}{F}(s)=\sum_{\rho\, :\, |\rho-(\sigma_1+iT)|\le 2\sigma_1}\frac{1}{s-\rho}+O(\log T).
\end{equation*}
Recall that $\sigma_1$ was defined before $\eqref{gec}$.
Here and elsewhere in this proof  the constants  in big-$O$ and $\ll$ notations may only depend on the subclass $\bar S$.
By the last equality,  the zero free region (\ref{ring}), and \eqref{FF} we get

\begin{equation*}
g(s)=\sum_{\rho\, :\, |\rho-(\sigma_1+iT)|\le 2\sigma_1\ \text{and}\atop |\rho-s_0|> r_j}\frac{1}{s-\rho}+O(\log T).
\end{equation*}
Using this expression in the integral for $a_n$  we obtain that
\begin{equation}\label{newan}
a_n
\ll
r_j^{-n}\log T \quad (n\ge1).
\end{equation}

Let us choose 

$$r=r_j^{1+\delta/3}.$$
 Clearly, the bounds \eqref{expta} are satisfied. By \eqref{Cauchy} and \eqref{newan}, for $|s-s_0|=r$, we have
$$
g(s)=a_0
 +O\left( r_j^{\delta/3} \log T\right).
$$ 
Hence, for $|s-s_0|=r$,  the expression \eqref{FF} gives
\begin{equation}\label{takingrealparts}
\Re \frac{F'}{F}(s)
=
\Re a_0+\sum_{\rho\, :\, |\rho-s_0|\le r_{j-1}}\frac{\sigma-\beta}{|s-\rho|^2}
+O\left(r_j^{\delta/3}  \log T\right).
\end{equation}
For $|\rho-s_0|\le r_{j-1}$, $|s-s_0|=r$, $1/2-(\Re s_0-1/2+ r_{j-1})\le\sigma\le1/2$, and large $T$, we have that
$
|\sigma-\beta|\le 4r_{j-1}$
and
$
|s-\rho|^2>r_j^{2+2\delta/3}/2.
$
Then by \eqref{1log2} 
we get

\begin{equation*}
\sum_{\rho\, :\, |\rho-s_0|\le r_{j-1}}\frac{\sigma-\beta}{|s-\rho|^2}\ll r_j^{\delta/3}\log T.
\end{equation*}
Consequently, by \eqref{takingrealparts},
\begin{equation}\label{isgasdino}
\Re \frac{F'}{F}(s)
=\Re a_0 +O\left(r_j^{\delta/3} \log T\right).
\end{equation}

The region \eqref{ring} is zero-free. Thus  $F(s)$ does not vanish on the circle $|s-s_0|=r$. By  instantiating  (\ref{logr11/2}) and \eqref{isgasdino} to a single $s$ on the intersection of $|s-s_0|=r$ and $\sigma=1/2$ we obtain that
\begin{equation}\label{rea}
\Re a_0=-\frac{d_F}2 \log T-\log Q+O\left(\frac1T\right)+O\left(r_j^{\delta/3} \log T\right).
\end{equation}

Hence,  for $|s-s_0|=r$ and $1/2-(\Re s_0-1/2+ r_{j-1})\le\sigma\le1/2$, 
\begin{equation}\label{final1}
\Re \frac{F'}{F}(s)
=-\frac{d_F}2 \log T -\log Q+O\left(\frac1T\right).
\end{equation}

 If  $|s-s_0|=r$ and $\sigma<1/2-(\Re s_0-1/2+ r_{j-1})$, then 

 \begin{equation*}
\sum_{\rho\, :\, |\rho-s_0|\le r_{j-1}}\frac{\sigma-\beta}{|s-\rho|^2}\le 0
\end{equation*}
and, in view of formulas \eqref{takingrealparts}, \eqref{rea},
\begin{equation}\label{final2}
\Re \frac{F'}{F}(s)
\le-\frac{d_F}2 \log T -\log Q+O\left(\frac1T\right).
\end{equation}
The expressions \eqref{final1} and \eqref{final2}, together with the zero free region \eqref{ring}, prove Lemma~\ref{crho}.
\end{proof}

\begin{proof}\hspace{-1.4mm}{\bf of Theorem \ref{prop}} \ \
Let

$$
R=\{s : |s-s_0|\le r\ {\text{and}}\ \sigma<1/2\},
$$
where $r$ is from Lemma \ref{crho}. To prove the theorem, it is enough to consider the difference in the number of zeros of $F(s)$ and $F'(s)$ in the region $R$.

 We consider the
change of $\arg F'/ F(s)$ along the appropriately indented boundary
$R'$ of the region $R$.  More precisely,  the left side
of $R'$ coincides with the circle segment $\{ s :  |s-s_0|=r, \sigma\le 1/2\}$. To
obtain  the right-hand side of the contour of $R'$, we take the
right-hand side boundary of $R$ and deform it to bypass the zeros of
$F(1/2+it)$ by left semicircles with an arbitrarily small radius.
In \cite[proof of Theorem 1.2]{gs} it is showed that on the right-hand side
 of $R'$ the inequality 
 \begin{equation}\label{logderf}
 \Re \frac{F'}{F}(s)<0
\end{equation}
is true. 
Then, in view of Lemma \ref{crho}, we have that the inequality \eqref{logderf} is valid on the whole contour $R'$. Therefore, the
change of $\arg F'/ F(s)$ along the contour $R'$ is less than $\pi$. This proves Theorem \ref{prop}.
\end{proof}

\begin{proof}\hspace{-1.4mm}{\bf of Theorem \ref{dertest}} \ \ The Riemann zeta-function is an element of degree $1$ of the extended Selberg class (Kaczorowski \cite{Kaczorowski06}). By  Trudgian \cite[Corollary 1]{Trudgian14} we see that, for large $T$, the Riemann zeta-function has less than $0.225\log T$ zeros in the strip $|t-T|\le1/T$. Thus in the formula \eqref{1log2} we choose $\varepsilon=0.17$ and $\delta=0.1$. Then Theorem \ref{dertest} follows from Theorem \ref{prop}. 
\end{proof}

\begin{proof}\hspace{-1.4mm}{\bf of Theorem \ref{cor}} \ \ 
We will use Theorem \ref{prop}.  Next we show that there is a subclass $\bar S$ such that $f(s,\tau)\in \bar S$  for all $\tau\in[0,1]$.  In view of the definition \eqref{fstau} of $f(s,\tau)$ we see that there are  constants $c$, $B$,
and $\sigma_1$ independent of $\tau$ for which the bounds \eqref{gec} and \eqref{TB} are valid.  By this and Jensen's theorem, similarly as in Titchmarsh \cite[Theorem 9.2]{Titchmarsh1986}, we get that there are  constants $\varepsilon$, $\delta$, and $\bar T$ independent of $\tau$ for which the zero number bound \eqref{1log2} is true. The function $f(s)=f(s,\tau)$ satisfies the functional equation (\cite[formula (3)]{gs})
\begin{equation}
  \label{eq:compfunctional}
  f(s) = 5^{-s + 1/2}2 (2 \pi)^{s - 1}\Gamma(1-s) \sin \left( \frac{\pi
      s}{2} \right) f(1 - s)
\end{equation}
which is independent of $\tau$, thus the  constants $\lambda_j$, $\mu_j$ are also independent of $\tau$. This proves the existence of required $\bar S$. Therefore in Theorem \ref{prop} with $F(s)=f(s,\tau)$ it is possible to choose $T_0$, which is  independent of $\tau$. Further in this proof we assume that $\Im(\rho_0)>T_0+10$.

We consider a zero trajectory $\rho(\tau)$ of $f(s,\tau)$ which satisfies $\rho(\tau_0)=\rho_0$. 
The two variable function $f(s,z)$ is holomorphic in a neighborhood of any

$$(s,z)\in \mathbb C^2\setminus \{(1,z) : z\in\mathbb C\}.$$
By conditions of the theorem we have that $\rho_0\ne1$, $f(\rho_0,\tau_0)=0$,
 \begin{equation}\label{dervatives}
\frac{\partial f(\rho_0, \tau_0)}{\partial s}=0,\quad\text{and}\quad\frac{\partial^2 f(\rho_0, \tau_0)}{\partial s^2}\ne0.
\end{equation} 
By \eqref{dervatives} and by the Weierstrass preparation theorem (Krantz and Parks \cite[Theorem 5.1.3]{KP2013}) there exists a polynomial 

$$p(s,\tau)=s^2+a_1(\tau)s+a_0(\tau),$$
where each $a_j(\tau)$ is a holomorphic function in a neighborhood of $\tau=\tau_0$ that vanishes at $\tau=\tau_0$, and there is a function $u(s,\tau)$ holomorphic and nonvanishing in some neighborhood $N$ of $(\rho_0,\tau_0)$ such that 
\begin{equation}\label{up}
f(s,\tau)=u(s,\tau)p(s,\tau)
\end{equation}
holds in $N$. Solving 
$s^2+a_1(\tau)s+a_0(\tau)=0$
we get
\begin{equation}\label{2sol}
s_{1,2}=s_{1,2}(\tau)=\frac{-a_1(\tau)\pm\sqrt{a_1(\tau)^2-4a_0(\tau)}}{2},
\end{equation}
where for the square-root we choose the branch defined by $\sqrt{1}=1$. Note that in the neighborhood $N$ the function $f(s,\tau)$ has no other zeros except those described by \eqref{2sol}.

Assume that the statement 1) of Theorem \ref{cor} is true.  Then in some neighborhood $U$ of $\tau=\tau_0$ the first part of trajectory $\rho(\tau)$ consists either of $\{s_1(\tau) : \tau<\tau_0, \tau\in U\}$ or of $\{s_2(\tau) : \tau<\tau_0, \tau\in U\}$. Similarly, the remaining part of trajectory $\rho(\tau)$ consists either of $\{s_1(\tau) : \tau>\tau_0, \tau\in U\}$ or of $\{s_2(\tau) : \tau>\tau_0, \tau\in U\}$.

  If  $\Re s_{1}(\tau)\ne1/2$ or $\Re s_{2}(\tau)\ne1/2$  for some $\tau$, then by the functional equation \eqref{eq:compfunctional}  we see that  $s_2(\tau)=1-\overline{s_1(\tau)}$. This and the condition {\it (iii)} give that 
\begin{equation}\label{notequal1}
s_1(\tau)\ne s_2(\tau),\quad \text{if} \quad \tau>0,\ \tau\in U.
\end{equation}
Thus $a_1(\tau)^2-4a_0(\tau)\ne0$ if $\tau>0$, $\tau\in U$. By the condition {\it (i)} we see that $\rho(\tau_0)=s_1(\tau_0)=s_2(\tau_0)$ is a double zero of $P(s)=P(s,\tau)$, therefore $a_1(\tau_0)^2-4a_0(\tau_0)=0$. Hence $a_1(\tau)^2-4a_0(\tau)$ is a non-constant holomorphic function. Then there is a neighborhood of $\tau=\tau_0$, where

\begin{equation}\label{notequal2}
s_1(\tau)\ne s_2(\tau),\quad \text{if} \quad \tau<0. 
\end{equation}

In view of formulas \eqref{dervatives}, the implicit function theorem (\cite[Theorem 2.4.1]{KP2013})) yields  the existence of $\eta>0$ and of a continuous function

 $$\tilde{\rho} : (\tau_0-\eta, \tau_0+\eta)\to\mathbb C,$$
such that $\tilde{\rho}(\tau_0)=\rho(\tau_0)=0$ and $f'_s(\tilde{\rho}(\tau), \tau)=0$. By this we get condition {\it (a)} of the second statement.

We assume that  $\eta>0$ is such that the set 

$$\{(\tilde{\rho}(\tau),\tau) : \tau\in  (\tau_0-\eta, \tau_0]\}$$
is a subset of the neighborhood $N$ (defined by \eqref{up}).
We have (\cite[Proposition 1.4]{gs}) that  $f'_s(1/2+it, \tau)=0$ implies $f(1/2+it, \tau)=0$. Then in view of 
 \eqref{notequal2}  we obtain that $\Re\tilde{\rho}(\tau)\ne1/2$ if $\tau\in  (\tau_0-\eta, \tau_0)$. By condition {\it (ii)} and by above there is a neighborhood of $(\rho_0,\tau_0)$, where $f(s,\tau)\ne0$ if $\tau<\tau_0$. Then condition {\it (b)} follows from Theorem~\ref{prop}.

Theorem \ref{prop} and condition {\it (iii)} lead to $\Re(\tilde{\rho}(\tau))<1/2$ if $\tau\in(\tau_0, \tau_0+\eta)$  and $\eta>0$ is sufficiently small. We get condition {\it (c)}. By this we proved that the  statement 1) implies the statement 2).

Assume the second statement of Theorem \ref{cor}. Then by applying Theorem \ref{prop} and reasoning similarly as above, we see  that from   the trajectories defined by \eqref{2sol} we can construct a trajectory $\rho(\tau)$ which satisfies conditions of the first statement. 
\end{proof}

{\it Acknowledgement.}
This research is funded by the European Social Fund according to the activity ‘Improvement of researchers’ qualification by implementing world-class R\&D projects’ of Measure  No. 09.3.3-LMT-K-712-01-0037.


\begin{thebibliography}{99}
\bibitem{garunkstistamosiunas}
Garunk\v{stis} R, Tamo\v{s}i\={u}nas, R. Zeros of the {L}erch zeta-function and of its derivative for equal parameters. To appear in Bulletin Mathématique de la Société des Sciences Mathématiques de Roumanie.
\end{thebibliography}
\end{document}